\documentclass[a4paper,11pt,centertags,reqno]{amsart}
\usepackage[latin1]{inputenc}
\usepackage[english]{babel}
\usepackage{pifont}
\usepackage{cite}
\usepackage{verbatim}
\usepackage{color}
\usepackage{hyperref}
\usepackage{extarrows}
\usepackage{graphicx,picture}
\usepackage{pgf,tikz}
\usetikzlibrary{arrows}
\usepackage{amsmath,amsthm}
\usepackage{comment}
\usepackage{pgfplots}

\usepackage[centering,margin=1in,nomarginpar,includeheadfoot]{geometry}

\newcommand{\diesis}{^\#}
\newtheorem{theo}{Theorem}[section]
\newtheorem{lemma}{Lemma}[section]
\newtheorem{prop}{Proposition}[section]

\theoremstyle{definition}
\newtheorem{definiz}{Definition}[section]
\newtheorem{rem}{Remark}[section]

\numberwithin{equation}{section}

\newcommand{\R}{\mathbb R}

\newcommand{\de}{\partial}
\newcommand{\aste}{\text{\ding{86}}}
\newcommand{\stella}{\text{\ding{73}}}
\newcommand{\conv}{\text{\ding{71}}}

\newcommand{\ds}{\displaystyle}
\DeclareMathOperator{\divergenza}{div}

\DeclareMathOperator{\dist}{dist}

\DeclareMathOperator{\Cof}{Cof}

\begin{document}
\title[Symmetrization with respect to the anisotropic perimeter]{Symmetrization with respect to the anisotropic perimeter and applications}
\author[F. Della Pietra, N. Gavitone]{
  Francesco Della Pietra$^{*}$ and Nunzia Gavitone$^{*}$
}
\date{}
\thanks{$^{*}$Universit\`a degli studi di Napoli Federico II,
Dipartimento di Ma\-te\-ma\-ti\-ca e Applicazioni ``R. Caccioppoli'',
Complesso di Monte Sant'Angelo, Via Cintia,
80126 Napoli, Italia
Fax:+390817662106. \\
Email: f.dellapietra@unina.it, nunzia.gavitone@unina.it}
\maketitle

\begin{abstract}
In this paper we introduce a new type of symmetrization, which preserves the anisotropic perimeter of the level sets of a suitable concave smooth function, in order to prove sharp comparison results for solutions of a class of homogeneous Dirichlet fully nonlinear elliptic problems of second order and for suitable anisotropic Hessian integrals.\\


\end{abstract}

\section{Introduction}

Let $\Omega$ be a bounded, strictly convex domain of $\R^{2}$. In a nowadays classical paper by Talenti \cite{ta81}, it was proved that if $u$ is a $C^{2}$ concave solution of the Monge-Amp\`ere equation
\begin{equation}
\label{euclideo}
\left\{
\begin{array}{ll}
	\det \nabla^{2} u =f(x) &\text{in }\Omega,\\
	u=0&\text{on }\de\Omega,
\end{array}
\right.
\end{equation}
with $f$ positive and sufficiently smooth function, then 
\begin{equation}
\label{talenti}
\tilde u(x) \le v(x), \quad x\in D,
\end{equation}
where $v$ is the positive concave solution of the problem
\begin{equation}
\label{euclideorad}
\left\{
\begin{array}{ll}
	\det \nabla^{2} v =f\diesis(x) &\text{in }D,\\
	v=0&\text{on }\de D.
\end{array}
\right.
\end{equation}
Here $D$ is the disk centered at the origin with the same perimeter of $\Omega$, $f\diesis$ is the spherically decreasing rearrangement of $f$, and $\tilde u$ is the spherically decreasing function in $D$ whose level sets have the same perimeter of the level sets of $u$. Hence, among all the problems of the type \eqref{euclideo} with prescribed perimeter of $\Omega$ and fixed rearrangement of $f$, problem \eqref{euclideorad} gives the ``maximal'' solution.
After this result, generalizations in several directions have been studied (see for example \cite{br,bntpoin,bt07,bt07bis,dpg2,dpg3,ga09,tr2,tso}). Moreover, it was proved in \cite{tr2}, \cite{tso} that a P\'olya-Szeg\"o type inequality for the Monge-Amp\`ere operator holds, namely
\begin{equation}
\label{poleucl}
\int_{\Omega} u\det D^{2} u\, dx \ge \int_{D} \tilde u \det D^{2} \tilde u\, dx.
\end{equation}
Hence, the symmetrization with respect to the perimeter decreases the Hessian integral. 

In this paper we take into account a class of fully nonlinear elliptic  anisotropic operators of the second order, which contains the classical  Monge-Amp\`ere operator. More precisely, we consider a sufficiently smooth norm $H$ of $\R^2$, and denote with $H^o$ its polar function (see Section 2 for the precise assumptions). If $u$ a smooth concave function, we take into account the following anisotropic Monge-Amp\`ere operator
\begin{equation}
\label{detintr}
\vspace{.1cm}
{\det}_{H}[u]:=\det \nabla^{2}_{\xi} F \,\det \nabla^{2} u,
\vspace{.1cm}
\end{equation}
where $F(\xi)=\frac 1 2 {H(\xi)^{2}}$, $\xi \in \R^{2}$, and $\nabla^{2}_{\xi}F(\xi)$ is the $2\times 2$ matrix of the second derivatives of $F$. We observe that when $H$ is  the Euclidean norm of $\R^{2}$, then the matrix $\nabla^{2}_{\xi} F$ reduces to the identity, and \eqref{detintr} reduces with the classical Monge-Amp\`ere operator.

Several kind of problems related to anisotropic operators have been largely studied in last years. We refer the reader, for example, to \cite{aflt,atw,and,bp,bfk,bf,ciasal,cfv,dpg4,dpg5,et,fvol,ja,wxarma}.

The aim of this paper is to prove a suitable generalization of the inequalities \eqref{talenti} and \eqref{poleucl} for the operator \eqref{detintr} using a new type of symmetrization. More precisely, given a concave smooth function $u$, the rearranged function we introduce preserves the anisotropic perimeter of its level sets, that is a suitable measure of the length of $\{u=t\}$ which takes into account the anisotropy $H$.  For this purpose a fundamental tool is a well-known anisotropic isoperimetric inequality (see for instance \cite{bu,aflt,fomu,dpf,dpg1}). Moreover, a key role is played by a relation between the operator \eqref{detintr} and the anisotropic curvature of the level sets of $u$ (see Sections 2 and 3).

The structure of the paper is the following. In Section 2, we state the main hypotheses on the norm $H$, and we recall the definition of anisotropic perimeter and curvature, proving a version of the Gauss-Bonnet Theorem in this setting. In Section 3, we introduce the anisotropic Monge-Amp\`ere operator, proving its connection with the anisotropic curvature. Moreover, we compute the operator \eqref{detintr} for functions which are symmetric with respect to $H^{o}$. 

The Section 4 is devoted to define the symmetrization with respect to the anisotropic perimeter and to prove the P\'olya-Szeg\"o inequality. Finally, in Section 5 we prove the quoted comparison result in the spirit of Talenti's inequality.

%
%
%

\section{Notation and preliminaries}
Throughout the paper we will denote by $H:\R^2\rightarrow [0,+\infty[$, a convex function such that
\begin{equation}\label{eq:lin}
  \alpha|\xi| \le H(\xi),\quad \forall \xi\in \R^2,
\end{equation}
and
\begin{equation}\label{eq:omo}
  H(t\xi)= |t| H(\xi), \quad \forall \xi \in \R^2,\; \forall t \in
  \R.
\end{equation}
Under these hypotheses it is easy to see that there exists $\beta\ge \alpha$ such that
\begin{equation}
\label{betabound}
  H(\xi) \le \beta|\xi|,\quad \forall \xi\in \R^2.
\end{equation}
Moreover, we assume that $H^{2}$ is strongly convex, that is $H\in C^{2}(\R^{2}\setminus\{0\})$ and the Hessian matrix $\nabla^{2}_{\xi} H^{2}$ is positive definite in $\R^{2}\setminus\{0\}$. 

The polar function $H^o\colon\R^2 \rightarrow [0,+\infty[$  of $H$ is defined as
\begin{equation}
  \label{eq:pol}
H^o(v)=\sup_{\xi \ne 0} \frac{\xi\cdot v}{H(\xi)}  
\end{equation}
where $\langle\cdot,\cdot\rangle$ is the usual scalar product of
$\R^2$. It is easy to verify that also $H^o$ is a convex function
which satisfies properties \eqref{eq:lin} and \eqref{eq:omo}. Furthermore, 
\[
H(v)=\sup_{\xi \ne 0} \frac{\xi\cdot v}{H^o(\xi)}.
\]
The set
\[
\mathcal W = \{  \xi \in \R^n \colon H^o(\xi)< 1 \}
\]
is the so-called Wulff shape centered at the origin. We put
$\kappa=|\mathcal W|$, where $|\mathcal W|$ denotes the Lebesgue measure
of $\mathcal W$. 
More generally, we denote with $\mathcal W_r(x_0)$
the set $r\mathcal W+x_0$, that is the Wulff shape centered at $x_0$
with measure $\kappa r^2$, and $\mathcal W_r(0)=\mathcal W_r$.

The strong convexity of $H^{2}$ implies that $\{\xi\in\R^{2}\colon H(\xi)< 1\}$ is strictly convex. This ensures that $H^{o}\in C^{1}(\R^{2}\setminus\{0\})$. Actually, the strict convexity of the level sets of $H$ is equivalent to the continuous differentiability of  $H^{o}$ in $\R^{2}\setminus\{0\}$  (see \cite{schn} for the details).

The following properties of $H$ and $H^o$ hold true
(see for example \cite{bp}):
\begin{gather}
H(\xi)=H_{\xi}(\xi)\cdot \xi,\quad H^{o}(\xi)=H^{o}_{\xi}(\xi)\cdot \xi,\quad \forall \xi \in
\R^2\setminus \{0\}\label{eul} \\
 H( H_{\xi}^o(\xi))=H^o(H_{\xi}(\xi))=1,\quad \forall \xi \in
\R^2\setminus \{0\}, \label{eq:H1} \\
H^o(\xi)  H_{\xi}( H_{\xi}^o(\xi) ) = H(\xi) 
H_{\xi}^o( H_{\xi}(\xi) ) = \xi,\quad \forall \xi \in
\R^2\setminus \{0\}. \label{eq:HH0}
\end{gather}
\subsection{Anisotropic perimeter}
Let $K$ be an  open bounded set of $\mathbb R^n$ with Lipschitz boundary. The
  perimeter of $K$ is defined as the quantity
\[
P_H(K) = \int_{ \partial K} H(\nu_K) d\mathcal H^1,
\]
where $\nu_{K}$ is the Euclidean outer normal to $\de K$. 
For example, if $K=\mathcal W_{R}$, then
\begin{multline*}
P_{H}(\mathcal W_{R})=\int_{\de \mathcal W_{R}} \frac{1}{|\nabla H^{o}(x)|} d\mathcal H^{1}= \frac 1 R \int_{\de \mathcal W_{R}} \frac{x\cdot \nabla H^{o}(x)}{|\nabla H^{o}(x)|} d\mathcal H^{1} = \\= 
\frac 1 R \int_{\de \mathcal W_{R}} x\cdot \nu\, d\mathcal H^{1} =
\frac 2 R \int_{\mathcal W_{R}} dx = 2\kappa R,
\end{multline*}
where in the above computations we used \eqref{eul} and the divergence theorem.

The anisotropic perimeter of a set $K$ is finite if and only if the usual Euclidean perimeter $P(K)$ is finite. Indeed, by properties \eqref{eq:omo} and \eqref{eq:lin} we have that
\[
\frac{1}{\beta} |\xi| \le H^o(\xi) \le \frac{1}{\alpha} |\xi|,
\]
and then
\begin{equation*}\label{eq:per}
\alpha P(K) \le P_H(K) \le \beta P(K).
\end{equation*}

An isoperimetric inequality for the anisotropic perimeter holds,
namely
\begin{equation}
  \label{isop}
  P_H(K)^{2} \ge 4\kappa |K|
\end{equation}
(see for example
\cite{bu,dpf,fomu,aflt}). We stress that in
\cite{dpg1} an isoperimetric inequality for the
anisotropic relative perimeter in the plane is studied. 

Moreover, if $K$ is a convex body of $\R^{2}$, and $\delta>0$, the following Steiner formulas hold (see \cite{and,schn}):
\begin{equation}
  \notag
|K+\delta \mathcal W|= |K| + P_H(K) \delta + \kappa \delta^2
\end{equation}
and
\begin{equation}
\label{st2}
P_H(K+\delta \mathcal W) = P_H(K) +
2 \kappa \delta.
\end{equation}

\subsection{Anisotropic curvature}
For the sake of simplicity, will assume the following conventional notation: given a smooth function $u$, then $\de_{x_{i}}u=u_{i}$, for $i=1,2$, and $\de_{x_i x_j}u =u_{ij}$, for $i,j=1,2$.

We recall the definition and some properties of anisotropic curvature for a smooth set. For further details we refer the reader, for example, to \cite{atw} and \cite{bp}.
\begin{definiz}
  Let
  $K\subset \R^2$ be a bounded open set with smooth boundary,
  $\nu_K(x)$ the unit outer normal at $x\in \partial K$, in the
  usual Euclidean sense. Let $u$ be a $C^2(K)$ function such that
  $=\{u>0\}$, $\de K = \{ u=0 \}$ and $\nabla u \ne (0,0)$ on $\de
  K$. Hence, $\nu_K=-\frac{\nabla u}{|\nabla u|}$ on $\de K$.
  The anisotropic outer normal $n_{K}$ is defined as
  \[
  n_K(x)= \nabla_{\xi} H(\nu_K(x))=\nabla H_{\xi}\left(-\frac{\nabla u}{|\nabla u|}\right),\quad x\in \partial K.
  \]
  By the properties of $H$ and $H^{o}$,
  \[
  H^o(n_K)=1.
  \]
  The anisotropic curvature $k_H$ of $\partial K$ is
  \[
  k_H(x)= \divergenza n_K(x)=\divergenza\left[ \nabla_{\xi}
    H\left(-\frac{\nabla u}{|\nabla u|}\right) \right], \quad x\in
  \de K.
  \]
  Being $\nabla_{\xi}H$ a $0$-homogeneous function, it follows that
  \begin{equation}
  \label{curvform2}
  k_{H}(x)=-\divergenza \big[ \nabla_{\xi} H(\nabla u)\big]= -\sum_{i,j} H_{\xi_{i}\xi_{j}}(\nabla u)\,u_{ij}\quad\text{on }\de K.  	
  \end{equation}
\end{definiz}
\begin{rem}
If $H(x)=(x_{1}^{2}+x_{2}^{2})^{1/2}$ is the Euclidean norm of $\R^{2}$, then the above definition coincides with the classical definition of curvature in the plane.
\end{rem}
\begin{rem}
We stress that if $K=\mathcal W_{\lambda}(x_{0})=\{x\in \R^{2}\colon H^o(x-x_0)<\lambda\}$, that is homothetic to the Wulff shape $\mathcal W$ and centered at $x_0\in \R^{2}$, the anisotropic outer normal at $x\in \de K$ has the direction of $x-x_0$. Indeed by the properties of $H$ it follows that
\[
n_K(x)= \nabla_\xi H\big(\nabla_{\xi} H^o(x-x_0) \big) =
\frac{1}{\lambda}(x-x_0),\quad x\in \de K
\]
(see Figure \ref{fig:normal} for an example). Moreover, computing the anisotropic curvature at $x\in \partial K$ we have that
\[
k_H(x)= \frac 1 \lambda.
\]
\end{rem}
\begin{center}
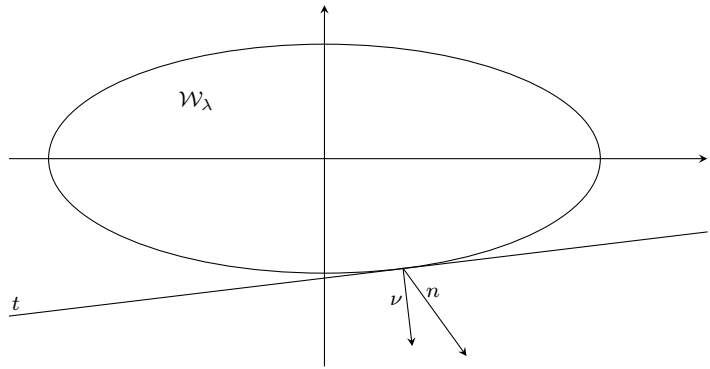
\begin{figure}[h]
\begin{tikzpicture}[line cap=round,line join=round,>=triangle
  45,x=1.0cm,y=1.0cm,>=stealth,scale=1.1]
\draw[->,color=black] (-3.77,0) -- (4.58,0);
\draw[->,color=black] (0,-2.5) -- (0,1.85);
\draw [rotate around={0:(0,0)}] (0,0) ellipse (3.3cm and 1.38cm);
\draw [domain=-3.77:4.58] plot(\x,{(-3.08--0.26*\x)/2.14});
\draw [->] (0.94,-1.32) -- (1.05,-2.26);
\draw [->] (0.94,-1.32) -- (1.7,-2.38);
\draw[color=black] (-1.53,.7) node {\scriptsize $\mathcal W_{\lambda}$};
\draw[color=black] (-3.69,-1.74) node {\scriptsize $t$};
\draw[color=black] (.86,-1.71) node {\scriptsize $\nu$};
\draw[color=black] (1.3,-1.61) node {\scriptsize $n$};
\end{tikzpicture}
\caption{ Here $H(x_{1},x_{2})=(
  {x_{1}^2}/{a^2}+ {x_{2}^2}/{b^2})^{1/2}$ and  $H^o(x_{1},x_{2})=(
  {a^2}{x_{1}^2}+ {b^2}{x_{2}^2})^{1/2}$. When $a\ne b$, the usual and the
  anisotropic outer normal are, in general, different.}\label{fig:normal}
\end{figure}
\end{center}
The anisotropic curvature and the anisotropic perimeter are related as follows. By computing the first variation of the perimeter (see \cite[Theorem 5.1]{bp}, or \cite[Section 2.6, formula (2.24)]{and})
we have that if $K$ has smooth boundary
\[
 \lim_{\delta\rightarrow 0^+}\frac{P_H(K+\delta \mathcal W)-P_H(K)}{\delta} = \int_{\de \Omega} k_{H}(x) H(\nu)d\mathcal H^{1}.
 \]
Combining this identity and formula \eqref{st2}, we obtain the following anisotropic version of the Gauss-Bonnet theorem for boundaries of smooth convex sets.
\begin{prop}
If $K$ is a convex, bounded open set such that $\de K\in C^{2}$, then
\begin{equation}
\label{gaussbonnet}
\int_{\de K} k_{H}(x)\,H(\nu)\,d\mathcal H^{1}= 2\kappa.
\end{equation}
\end{prop}

\section{Anisotropic Monge-Amp\`ere operator}
Given any $\xi=(\xi_{1},\xi_{2})\in \R^{2}$, we denote by $F(\xi)$ the function 
\[
F(\xi)=\frac 1 2 {H(\xi)^{2}},
\]
and by $\nabla^{2}_{\xi}F(\xi)$ the $2\times 2$ matrix of the second derivatives of $F$. Hence its components are
\[
F_{\xi_{i}\xi_{j}} =
H_{\xi_{i}}H_{\xi_{j}}+
H H_{\xi_{i}\xi_{j}},\quad\text{for }i,j=1,2.
\]
Let $u$ be a smooth function, and consider the fully nonlinear operator $u \mapsto A[u]=A(\nabla u, \nabla^{2}u)$ defined as
\[
A[u] = \nabla_{x} \big[ F_{\xi}(\nabla u) \big]= \nabla^{2}_{\xi}F(\nabla u)\, \nabla^{2}u.
\]
We will take into account equations whose principal part is the following:
\[
{\det}_{H}[u]:=\det A[u]=\det \nabla^{2}_{\xi} F \,\det \nabla^{2} u.
\]
\begin{rem}
	We point out that if $H(x)=(x_{1}^{2}+x_{2}^{2})^{1/2}$ is the Euclidean norm of $\R^{2}$, then the matrix $\nabla^{2}_{\xi} F$ reduces to the identity, and the operator $\det A[u]$ coincides with the classical Monge-Amp\`ere operator.
\end{rem}
We will consider a convex, bounded, open set $\Omega\in \R^{2}$ with $C^{2}$ boundary, and functions belonging to the class 
\[
\Phi_{0}(\Omega)\colon= \left\{u\colon\Omega\to \R\, \big|\, u\in W^{2,2}(\Omega)\cap C^1(\bar \Omega),\, u \equiv 0 \text{ on }\de \Omega, u
\text{ concave in }\Omega
 \right\}.
\]
Being $F$ strongly convex, the functions $u\in \Phi_{0}(\Omega)$ are admissible in order to have that $\det_{H}$ is elliptic. Obviously, $u\in \Phi_{0}(\Omega)$ is either positive in $\Omega$, or $v\equiv 0$ in $\bar \Omega$.

Let us denote by 
\[
(S^{ij}(B))_{ij}= \Cof B =
\begin{pmatrix}
b_{22}& -b_{21}\\
-b_{12}& b_{11}
\end{pmatrix}
\] 
the cofactor of the matrix $B=(b_{ij})$. Observe that
\[
\sum_{i,j}S^{ij}(A[u])F_{\xi_{i}}u_{j} = 
\nabla_{\xi} F
 \Cof A[u]
\cdot \nabla u.
 \]

In \cite{ciasal} the following integration by parts formula is proved.

\begin{lemma}
\label{ciasaprop}
Let $u\in W^{2,2}(\Omega)\cap C^{1}(\bar \Omega)$, with $\Omega$ bounded open set such that $\de \Omega\in C^{1}$, and $u=0$ on $\de \Omega$. Then 
\begin{equation}
\label{ciasa}
\int_{\Omega} u\, {\det}_{H}[u]\,dx =-\frac 1 2 \sum_{i,j}\int_{\Omega} S^{ij}(A[u])\, F_{\xi_{i}}(\nabla u)\,u_{j}\,dx.
\end{equation}
\end{lemma}

For a function $u\in \Phi_{0}(\Omega)$, we denote the set $E_{u}$ as the following:
\begin{equation}
\label{Eu}
E_{u}=\{x\in\Omega\colon 0\le u(x)<{\max}_{\bar\Omega}u\}.
\end{equation}
Being $u\in C^{1}(\bar \Omega)$, $E_{u}$ is an open set.
\begin{theo}
\label{curvani}
Let $u\in \Phi_{0}(\Omega)\cap C^{2}(E_{u})$, $u\not\equiv 0$, and take $x\in \bar\Omega$ such that $u(x)\in [0,\max_{\bar\Omega}u[$.
\end{theo}
Then
\begin{equation}
\label{curvform}
k_{H}(x)= -H(\nabla u(x))^{-3} \sum_{i,j} \, S^{ij}(A[u(x)])\,F_{\xi_{i}}(\nabla u(x))\, u_{j}(x),
\end{equation}
where $k_{H}(x)$ is the anisotropic curvature of the level set $\{y\in\Omega\colon u(y)=u(x)\}$ at the point $x$.
\begin{proof}
First of all, we observe that $|\nabla u(x)|\ne 0$ by the regularity and the concavity of $u\not\equiv 0$. Denoting $H_{\xi_{i}}=H_{i}$, $H_{\xi_{i}\xi_{j}}=H_{ij}$, and using the Einstein summation convention, we have
\[
(S^{ij}(A[u]))_{ij} = 
\begin{pmatrix}
H_{2} H_{m}\, u_{2m} & -H_{2} H_{m}\,u_{1m} \\ 
-H_{1} H_{m}\,u_{2m} & H_{1} H_{m}\, u_{1m}
\end{pmatrix}+
H
\begin{pmatrix}
 H_{2m}\, u_{2m} & - H_{2m}\,u_{1m} \\ 
-H_{1m}\,u_{2m} & H_{1m}\, u_{1m}
\end{pmatrix}.
\]
Then, recalling that $k_{H}=-H_{ij}u_{ij}$, we have
\[
\begin{split}
S^{ij}&(A[u])F_{\xi_{i}}(\nabla u)u_{j}\\&= H^{2}(\nabla u)
\left[ H_{2m}u_{2m} H_{1}u_{1} + H_{1m}u_{1m}H_{2}u_{2}
-H_{2m}u_{1m}H_{1}u_{2}-H_{1m}u_{2m}H_{2}u_{1}\right]\\&=
H^{2}(\nabla u)\left[-k_{H}(H_{1}u_{1}+H_{2}u_{2})+\right.\\
&\qquad\quad\left.
-H_{1}u_{11}( H_{11} u_{1} + H_{12} u_{2})-H_{1}u_{12}(H_{21}u_{1}+H_{22}u_{2})+\right.\\
&\qquad\quad\left.-H_{2}u_{12}(H_{11}u_{1}+H_{12}u_{2})-H_{2}u_{22}(H_{21}u_{1}+H_{22}u_{2})
\right]
\\&=
H^3(\nabla u)\, k_{H}(x),
\end{split}
\]
where last equality follows from the 1 and 0 homogeneity of $H$ and $H_{\xi}$, respectively, being $\nabla_{\xi} H_{i}(\nabla u) \cdot \nabla u=0$. 
\end{proof}
\begin{rem}
We stress that if $H(x)=(x_{1}^{2}+x_{2}^{2})^{1/2}$ is the Euclidean norm of $\R^{2}$, then the identity \eqref{curvform} reduces to the well-known formula for the Euclidean curvature of the level sets of $u$:
\[
k(x)= -|\nabla u|^{-3} 
S^{ij}(\nabla^{2}u)\, u_{i}\, u_{j},
\]
with 
\[(S^{ij}(\nabla^{2}u))_{ij}= 
\begin{pmatrix}
u_{22} & -u_{12} \\ 
-u_{12} & u_{11} 
\end{pmatrix}.
\]
\end{rem}
The following Reilly-type inequality for the anisotropic determinant holds.
\begin{prop}
Let $u\in \Phi_{0}(\Omega)\cap C^{2}(E_u)$, where $E_u$ is defined in \eqref{Eu} and  $\Omega$ is a bounded convex open set such that $\de \Omega\in C^{2}$. Then 
\begin{equation}
\label{reilly}
\int_{u>t} {\det}_{H}[u]\, dx = \frac 1 2 \int_{u=t} k_{H}(x) \frac{H(\nabla u)^{3}}{|\nabla u|} d\mathcal H^{1}, \quad  t \in [0,\max_{\bar\Omega}u[.
\end{equation}
\end{prop}
\begin{proof}
Let $t\in [0,\max_{\bar\Omega}u[$, and $h>0$ sufficiently small, and apply Lemma \ref{ciasaprop} to the functions $u-t$ and $u-t-h$ in the sets $\{u>t\}$ and $\{u>t+h\}$ respectively. By subtracting, and recalling \eqref{curvform} we have
\begin{align*}
\int_{t< u\le t+h} (u-t){\det}_{H}[u] dx - h\int_{u>t+h} {\det}_{H}[u] dx &= \frac 1 2\int_{t< u\le t+h} k_{H}(x) H(\nabla u)^{3} dx=\\[.1cm]
&= \frac 1 2 \int_{t}^{t+h} d\tau \int_{u=\tau} k_{H}(x) \frac{H(\nabla u)^{3}}
{|\nabla u|} d\mathcal H^{1},
\end{align*}
where last equality follows from the coarea formula.
Hence, dividing for $h$ and passing to the limit, we easily get \eqref{reilly}.
\end{proof}
\subsection{The anisotropic Monge-Amp\`ere operator for radial functions}
Let $v(x)=w(H^{o}(x))=w(r)$, $x\in \mathcal W_{R}$, $r=H^{o}(x)$. We compute the operator $\det_{H}$ on $v$. We have:
\[
\nabla v(x)=w'(r) H^{o}_{\xi}(x).
\]
Then, using the homogeneity of $H$ and properties \eqref{eq:H1} and \eqref{eq:HH0}, it follows that
\[
F_{\xi}(\nabla v(x))= w'(r)\,H(H_{\xi}^{o}(x))\,H_{\xi}(H_{\xi}^{o}(x)) = w'(r) \frac{x}{H^{o}(x)}.
\]
Then
\[
A[v] = 
\begin{pmatrix}
w'' \frac{x_{1} H^{o}_{\xi_{1}}}{r}+\frac{w'}{r^{2}}\left(r-x_{1} H_{\xi_{1}}^{o}\right) & \frac{x_{2} H_{\xi_{1}}^{o}}{r}\left(w''-\frac{w'}{r}\right) \\[.3cm]
\frac{x_{1} H_{\xi_{2}}^{o}}{r}\left(w''-\frac{w'}{r}\right) & w'' \frac{x_{2} H^{o}_{\xi_{2}}}{r}+\frac{w'}{r^{2}}\left(r-x_{2} H_{\xi_{2}}^{o}\right)
\end{pmatrix}
\]
Then, computing the determinant of $A[v]$ and using \eqref{eul}, a 
straightforward computation leads to
\begin{equation}
\label{rad}
\det A[v] = \frac{w' w''}{r}=\frac{[(w')^{2}]'}{2r}.
\end{equation}
Hence the function
\[
	v(x)=\sqrt 2\int_{H^{o}(x)}^{R} \left(\int_{0}^{s} r\,g(r)\,dr	
	\right)^{\frac 1 2}ds
\]
is such that $v \in \Phi_0(\mathcal{W}_R) \cap C^2(\mathcal{W}_R\setminus\{0\})$ and $v$ is the unique anisotropic radially symmetric function  such that
\[
\left\{
\begin{array}{ll}
\det_{H}[v]= g(H^{o}(x)) & \text{ a.e. in }\mathcal W_{R}\setminus\{0\}\\
v=0 &\text{ on }\de\mathcal W_{R}.\\
\end{array}
\right.
\] 

\section{Symmetrization with respect to the anisotropic perimeter}
Now we recall some basic definition on rearrangements and convex symmetrization. Moreover  for a given function $u$, we introduce a new kind of symmetrization which preserves the anisotropic perimeter of the level sets of $u$.

Let $\Omega$ be a bounded open set, and $u\colon \Omega\rightarrow \R$ a  measurable function. We will adopt the following notation:
\[
\Omega_t =\{x \in \Omega \colon |u(x)|>t \} \text{ and } \Sigma_t= \de\Omega_t=  \{x \in \Omega \colon |u(x)|=t \}.
\]  
Moreover, $\mu(t)=|\Omega_{t}|$, $t\ge 0$ is the distribution function of $u$. 

The {decreasing rearrangement} of $u$ is the map
$u^*:\,[0,\infty[\rightarrow \R$ defined by
    \begin{equation*}
        u^*(s):=\sup\{t\in\R:\mu(t)>s\}.
    \end{equation*}
The function $u^*$ is the generalized inverse of $\mu$. 

Following \cite{aflt}, the convex symmetrization of $u$ is the
function $u^\conv(x)$, $x\in \Omega^\conv$ defined by:
\begin{equation*}
    u^\conv(x)=u^*(\kappa H^o(x)^2),
\end{equation*}
where $\Omega^\conv$ is a set homothetic to the Wulff shape having
the same measure of $\Omega$.

Now suppose that $\Omega$ is a convex set of $\R^2$ and let $u\ge 0$ be a measurable function with convex level sets.  For  $t\in [0,\max_{\bar \Omega} u]$ the anisotropic perimeter of the level set 
$\Omega_{t}$ is denoted with 
\begin{equation}
\label{lambda}
\lambda_H(t)= P_{H}(\Omega_{t}).
\end{equation}
It is well-known that $u$, $u^{\conv}$ and $u^*$ are equimeasurable.

\begin{definiz}
\label{asim}
The rearrangement of $u$ with respect to the anisotropic perimeter is the function $s\in [0,P_{H}(\Omega)]\mapsto u^\aste(s)\in [0,\max_{\bar\Omega}u]$ defined as
\[
u^{\aste}(s)=\sup\{t\ge 0\colon \lambda_{H}(t)\ge s\}.
\]
Moreover, we define the anisotropic radial symmetrand of $u$ with respect to the anisotropic perimeter the function
\[
u^{\stella}(x)=u^{\aste}\big(2\,\kappa\,H^{o}(x)\big),\quad x\in \Omega^{\stella},
\]
where $\Omega^{\stella}$ is the set homothetic to the Wulff shape $\mathcal W$ such that $P_{H}(\Omega^{\stella})=P_{H}(\Omega)$. More precisely, $\Omega^{\stella}=\mathcal W_{R}$, with $R=\frac{P_{H}(\Omega)}{2\kappa}$. 
\end{definiz}

From now on, we will suppose that $u\in \Phi_{0}(\Omega) \cap C^2(E_u)$, where $E_u$ is defined in \eqref{Eu}.

The  functions $ u^{\aste}$ and $u^{\stella}$ have the following properties:
\begin{itemize}
\item[(i)] $u^{\aste}$ is a concave and decreasing function in $[0,P_H(\Omega)]$;
\item[(ii)]  $u^{\stella}(x)$  is symmetric and decreasing with respect to $H^{o}$;
\item[(iii)] The sets   $\{u^\stella>t\}$ are homothetic to the Wulff shape such that $P_{H}(\{u^\stella>t\})=P_{H}(\{u>t\})$.
\item[(iv)] $u^{\aste}(\lambda_H(t))=t$.
\end{itemize} 
If $u\in \Phi_{0}(\Omega)$, the coarea formula gives that
\begin{equation}
\label{muprimo}
\mu'(t)= -\int_{\Sigma_{t}}\frac{1}{|\nabla u|}d\mathcal H^{1},\quad t\in[0,{\max}_{\bar\Omega}u[.
\end{equation}
Moreover, we have the following result.
\begin{prop}
\label{proplambda}
If $u\in \Phi_{0}(\Omega)\cap C^2(E_u)$, the function $\lambda_H(t)$ defined in \eqref{lambda} is strictly decreasing in $[0, \max_{\Omega} u]$, it is differentiable in $[0,\max_{\bar\Omega} u[$ and its derivative is
\begin{equation}
\label{proplambdaeq}
\lambda_{H}'(t) = -\int_{\Sigma_{t}} \frac{k_{H}(x)}{|\nabla u|}\, d\mathcal H^{1}.
\end{equation}
\end{prop}
\begin{proof}
Using the homogeneity of $H$ and the divergence theorem we get
\begin{multline*}
	\lambda_{H}(t)= \int_{\Sigma_{t}} H(\nu_{\Omega_{t}})d\mathcal H^{1}=
	-\int_{\Sigma_{t}} \nabla_{\xi} H(\nabla u) \cdot \nu_{\Omega_{t}} d\mathcal H^{1}= \\ =
	-\int_{\Omega_{t}} \divergenza \nabla_{\xi}H(\nabla u)\, dx= \int_{\Omega_{t}} k_{H}(x) dx.
\end{multline*}
Hence, being $|\nabla u|\ne 0$ on $\Sigma_{t}$, for any $t\in [0,\max_{\Omega}u[$,  by the coarea formula we obtain, for $t\in [0,\max_{\bar\Omega} u[$, that
\[
\frac 1 h \left[\lambda_{H}(t)-\lambda_{H}(t+h)\right] = \frac 1 h \int_{\{t< u\le t+h\}} k_{H}(x)\,dx= \frac 1 h \int_{t}^{t+h} \int_{\Sigma_{t}}\frac{k_{H}(x)}{|\nabla u|} d\mathcal H^{1}.
\]
Passing to the limit, we get \eqref{proplambdaeq}.
\end{proof}
As a consequence of Proposition \ref{proplambda}, we have that
\begin{prop} 
The function $u^{\aste}$ belongs to $C^{0,1}(]0,P_{H}(\Omega)])$, and there exists a positive constant $C>0$ such that
\[
0\le -(u^{\aste})'(s) \le C, \text{ for any }  s \in ]0,P_{H}(\Omega)].
\]
\end{prop}
\begin{proof}
Let $t\in [0,\max_{\bar\Omega}u[$.  By Proposition \ref{proplambda}, \eqref{betabound} and formula \eqref{gaussbonnet}, we have
\begin{gather}
\begin{split}
\label{reg}
-\lambda_{H}'(t) =& \int_{\Sigma_{t}} \frac{k_{H}(x)}{|\nabla u|}\, d\mathcal H^{1}\ge \frac{1}{\max_{\Omega}|\nabla u|}\int_{\Sigma_{t}} k_{H}(x)\, d\mathcal H^{1}\\
&\ge \frac{1}{\beta\max_{\bar\Omega}|\nabla u| }\int_{\Sigma_{t}} k_{H}(x) H(\nu)\, d\mathcal H^{1}\\
&=\frac{2 \kappa}{\beta\max_{\bar\Omega}|\nabla u| }.
\end{split}
\end{gather}
Being $u^{\aste}(\lambda_H(t))=t$ then 
\[
\lambda_{H}'(t)=\frac{1}{\frac{d}{dt}u^{\aste}(\lambda_H(t))}.
\]
Substituting in \eqref{reg} we get the thesis.
\end{proof}
The main difference between the symmetrand  of $u$ with respect to the anisotropic perimeter $u^{\stella}$ and the convex symmetrand of $u$,  $u^{\conv}(x)$,  is that, in general, the first one increases the Lebesgue norms of $u$. Indeed, we have the following. 
\begin{prop}
Let be $u \in \Phi_0(\Omega)\cap C^2(E_u)$ and $u^{\stella }$ as in Definition \ref{asim}. Then
\[ 
\|u\|_{L^p(\Omega)} \le \|u^{\stella}\|_{L^p(\Omega^{\stella})}, \quad 1\le p <+\infty,
\]
and
\[
\|u\|_{L^{\infty}(\Omega)}=\|u^{\stella}\|_{L^{\infty}(\Omega^{\stella})}.
\]
\end{prop}
\begin{proof}
It is enough to observe that, by the anisotropic isoperimetric inequality \eqref{isop} we have 
\[
|\{u>t\}|=\mu(t)\le  \frac{\lambda_H^2(t)}{4 \kappa}=\frac{P^2_H(\{u^{\stella}>t\})}{4 \kappa}=|\{u^{\stella}>t\}|.
\]
\end{proof}

In order to prove a P\'olya-type inequality for the symmetrization with respect to the anisotropic perimeter we need the following definition.
\begin{definiz}[Anisotropic Hessian integral]
Let $u \in \Phi_0(\Omega)\cap C^2(E_u)$. Then the anisotropic Hessian integral is
\[
I_H[u,\Omega]=\int_{\Omega} u\, {\det}_{H}[u]\,dx 
\]
\end{definiz}
\begin{rem}
By Theorem \ref{curvani} and the identity \eqref{ciasa}, for $u \in \Phi_0(\Omega)\cap C^2(E_u)$ the anisotropic Hessian integral can be written also as follows
\begin{equation}
\label{hesint}
I_H[u,\Omega]=\int_{\Omega} u\, {\det}_{H}[u]\,dx =\frac{1}{2}\int_{\Omega} k_H(x)\, H^3(\nabla u) \,dx.
\end{equation}
\end{rem}
When we consider anisotropic radially symmetric function $v(x)=w(H^o(x))$, the anisotropic Hessian integral, recalling \eqref{rad}, is naturally defined as follows.
\begin{definiz}
Let be $v$ a concave function in $C^{0,1}(\mathcal W_R)$, such that $v$ vanishes on $\de \mathcal W_R$ and  $v(x)=w(H^o(x))=w(r)$. Then 
\begin{equation}
\label{intrad}
I_H[v,\mathcal W_R]= \kappa \int_0^R |w'(t)|^3 \, dt.
\end{equation}
\end{definiz}
In particular, for $v(x)=u^{\stella}(x)=u^{\aste}(2\kappa H^{o}(x))$, $x\in\Omega^{\stella}$, $P_H(\Omega)=2\kappa R$ performing a change of variable we have that for 
\begin{equation}
\label{hesintrad}
	I_H[u^{\stella},\Omega^{\stella}]= 4\kappa^{3} \int_{0}^{P_{H}(\Omega)}
	|(u^{\aste})'(s)|^{3}ds.
\end{equation}

The following P\'olya-Szeg\"o inequality for anisotropic Hessian integral holds:
\begin{theo}
Let be $u \in \Phi_0(\Omega)\cap C^2(E_u)$. Then
\begin{equation}
\label{pzeq}
I_H[u,\Omega]\ge I_H[u^\stella,\Omega^\stella].
\end{equation}
Moreover, if $u$ is strictly concave, then the equality in \eqref{pzeq} holds if and only if, up to a translation, $\Omega=\Omega^{\stella}$ and $u=u^\stella$.
\end{theo}
\begin{proof}
Using \eqref{hesint} and the coarea formula, we get that, for 
$M=\max_{\bar\Omega}u$, 
	\begin{equation}
	\label{pspass}
	I_H[u,\Omega]=\frac{1}{2} \int_{\Omega} k_H(x)\, H^3(\nabla u) \,dx	
	= \frac{1}{2}\int_0^M dt \int_{{u=t}} H^{3}(\nabla u) \frac{k_{H}(x)}{|\nabla u|}d\sigma.
	\end{equation}
	Now observe that, by the H\"older inequality
	\begin{equation}
	\label{hold}
	\int_{u=t} H(\nabla u) \frac{k_{H}(x)}{|\nabla u|}d\sigma \le
	\left(\int_{u=t} H^{3}(\nabla u) \frac{k_{H}(x)}{|\nabla u|}d\sigma\right)^{\frac 1 3}
	\left(\int_{u=t}  \frac{k_{H}(x)}{|\nabla u|}d\sigma\right)^{\frac 2 3},
	\end{equation}
	then using the homogeneity of $H$ and formulas \eqref{gaussbonnet} and \eqref{proplambdaeq}, we have
	\[ 
	\int_{u=t}H^{3}(\nabla u)  \frac{k_{H}(x)}{|\nabla u|}d\sigma \ge \dfrac{\ds\left(\int_{u=t} k_{H}(x)H(\nu)d\sigma\right)^{3} }
	{\ds\left(\int_{u=t}  \frac{k_{H}(x)}{|\nabla u|} d\sigma\right)^{2}} = \frac{8\kappa^{3}}{(-\lambda'_{H}(t))^{2}} = 8\kappa^{3} [(-u^{\aste})'(\lambda_{H}(t))]^{2}.
	\]
	Hence applying the above inequality in \eqref{pspass}, performing the change of variable $s=\lambda_{H}(t)$ and recalling \eqref{hesintrad} we get \eqref{pzeq}. 
	
	Now suppose that $u$ is strictly concave in $\Omega$, and that equality in \eqref{pzeq}  holds. Then \eqref{hold} becomes an equality, hence
	\begin{equation*}
	\left.H(\nabla u)\right|_{\{u=t\}} = c(t), \quad t\in [0,{\max}_{\bar \Omega}u[. 
	\end{equation*}
	this implies, recalling \eqref{muprimo}, that
	\[
	\lambda_{H}(t)=\int_{\{u=t\}} \frac{H(\nabla u)}{|\nabla u|} d\sigma =-c(t)\mu'(t), \quad t\in[0,{\max}_{\bar \Omega}u[,
	\]
	and, by \eqref{proplambdaeq} and \eqref{gaussbonnet}, that
	\[
	\lambda_{H}'(t)=-\frac{2\kappa}{c(t)}, \quad t\in[0,{\max}_{\bar \Omega}u[.
	\]
	Hence from the two equalities above we have
	\[
	\lambda_{H}'(t)\lambda(t)= 2\kappa\,\mu'(t), \quad t\in[0,{\max}_{\bar \Omega}u[.
	\]
	Integrating, and recalling that $u$ is strictly concave in $\Omega$, there is a unique point where the function $u$ achieves its maximum, then we can integrate the above equality, obtaining that
	\[
	\lambda_{H}(t)^{2}=4\kappa\, \mu(t).
	\]
	Hence, equality occurs in the anisotropic isoperimetric inequality for all the level sets of $u$. Then, for any $t\in [0,{\max}_{\bar \Omega}u]$, the set $\{u>t\}$ is, up to a translation, homothetic to the Wulff shape. In particular, $\Omega=\Omega^{\stella}$. Together with the fact that $H(\nabla u)$ is constant on $\{u=t\}$, it is possible to proceed as in \cite{fvol}, obtaining that all the level sets have the same center and, up to a translation, $u=u^{\stella}$. 
\end{proof}
\begin{rem}
\label{esempio}
We observe that if we do not assume that $u$ is strictly concave, the equality sign can occurs in the inequality \eqref{pzeq} also if $u$ is not radial and $\Omega$ is not a Wulff shape. For the sake of simplicity, we give an example in the Euclidean case, with $H(x)=|x|$. Let us consider a strictly convex, bounded open set $\Omega_{0}$ with $C^{2}$ boundary, and let $\Omega$ be the set $\Omega_{0} + \delta D$, where $D$ is the unit disk of $\R^{2}$ centered at the origin, and $\delta>0$. Let us consider the function
\[
u(x)=\delta^{3}-d(x)^{3}, \quad x\in \Omega,
\] 
where $d(x)=\dist (x,\Omega_{0})=\inf_{z\in \Omega_{0}} |x-z|$, with $x\in \R^{2}$ (see Figure \ref{fig}). Then the convexity of $\Omega_{0}$ implies that $d$ is a convex function. Moreover, the smoothness of the boundary of $\Omega_{0}$ guarantees that $d$ is $C^2(\R^{2}\setminus \Omega_{0})$. Finally, $|\nabla d|=1$ in $\R^{2} \setminus \bar \Omega_{0}$ (for the properties of the distance function we refer the reader, for example, to \cite{gt} and \cite{rock}). 
Hence $u\in \Phi_{0}(\Omega)\cap C^2(\Omega)$, $|\nabla u|$ is constant on every level set of $u$. Hence, being $k_{H}(x)$ positive on every level set of $u$, the inequality \eqref{hold} becomes an equality. Then also in \eqref{pzeq} the equality sign holds, even if $u$ is not radially simmetric and $\Omega$ is not a ball. 

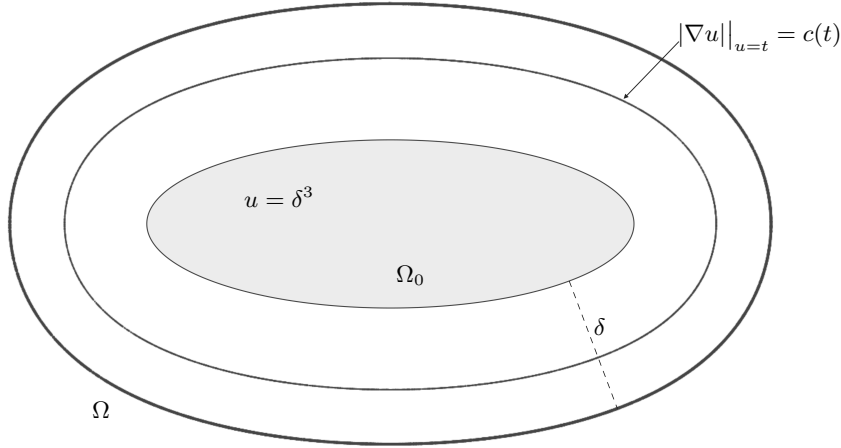
\begin{figure}[h]
\label{fig}
\includegraphics[scale=.6]{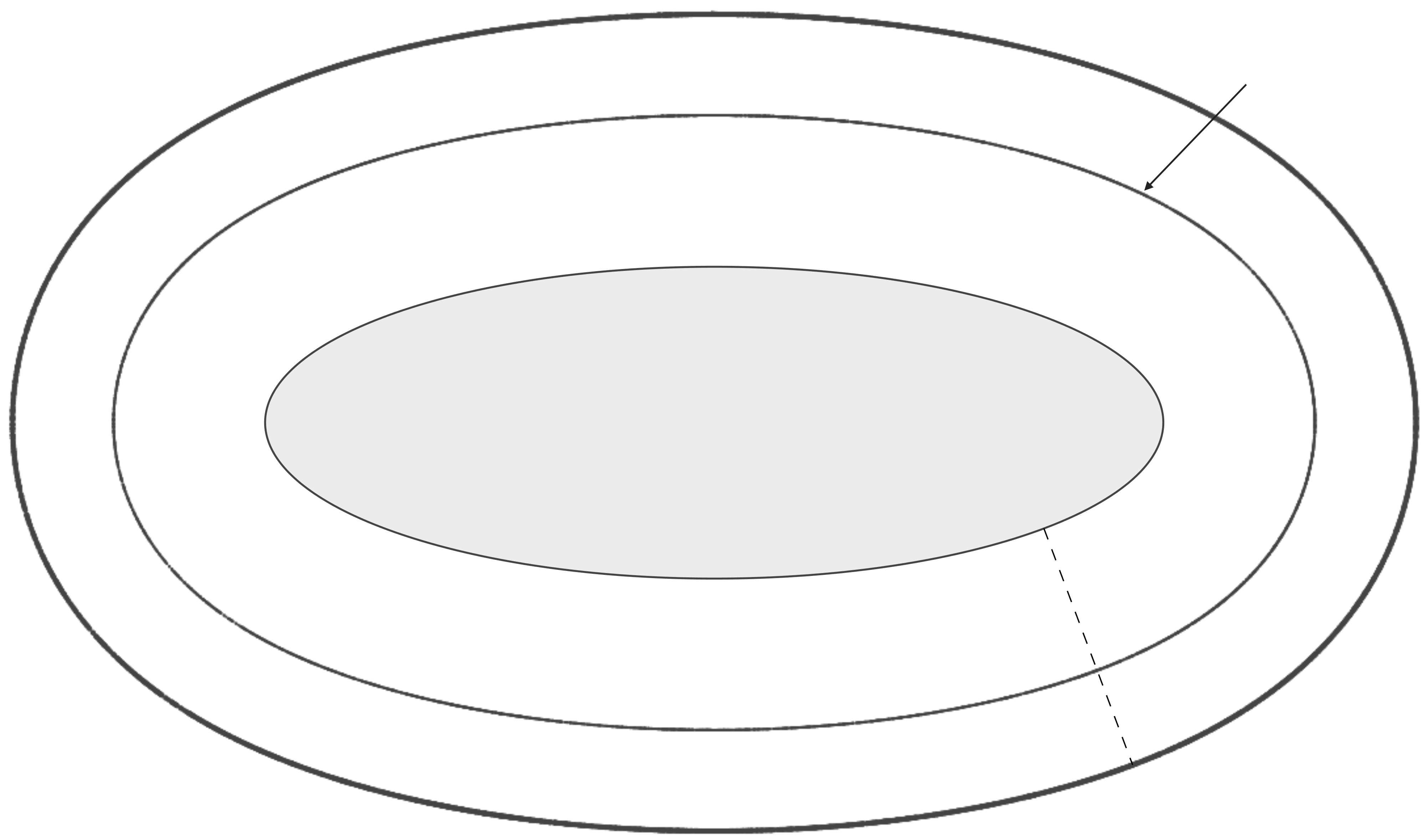}
  \put(-9cm,0.4cm){\footnotesize $\Omega$}
  \put(-2.4cm,1.5cm){\footnotesize $\delta$}
  \put(-5cm,2.2cm){\footnotesize $\Omega_{0}$}
    \put(-7cm,3.2cm){\footnotesize $u=\delta^{3}$}
   \put(-1.3cm,5.4cm){\footnotesize $|\nabla u|\big|_{u=t}=c(t)$} 
\caption{An example of the set $\Omega=\Omega_{0}+\delta D$ of Remark \ref{esempio}. In $\Omega_{0}$ the function $u$ is constant, while on the curve $\{u=t\}$, $0\le t\le\delta^{3}$, we have $|\nabla u|=c(t)$.}
\end{figure}
\end{rem}

\section{Comparison results}
In this section we  use the symmetrization  with respect to the anisotropic perimeter  to prove comparison results between the solutions of suitable fully nonlinear elliptic equations. Let us consider the following problem involving the anisotropic Monge-Amp\`ere operator:
\begin{equation}
\label{pb1det}
\left\{
\begin{array}{ll}
\det_H[u]=f(x) &\text{in } \Omega \\
u=0 &\text{on } \de\Omega,
\end{array}
\right.
\end{equation}
where $\Omega$ is a convex, bounded, smooth open set of $\R^2$ and $f>0$ belongs to $C^{0,\alpha}(\Omega)$.

We will consider strong solutions of problem \eqref{pb1det}, namely functions $u \in \Phi_0(\Omega)$ which satisfy
\[
{\det}_{H}[u]=f(x) \text{ a.e. in }\Omega.
\]
We stress that the positivity of $f$ ensures that the function  $u $ has not flat zones, and, by the concavity of $u$, $\nabla u=0$ only where $u$ attains its maximum.
\begin{rem}
\label{reg}
In the open set $E_{u}$, defined in \eqref{Eu}, the operator ${\det}_{H}[\,\cdot\,]$ is continuous. Then a strong solution $u\in \Phi_{0}(\Omega)$ is a viscosity solution in $E_{u}$ (see \cite[Corollary 3]{libon}). Moreover, if $x\in E_{u}$, 
then the equation in \eqref{pb1det} can be written as
\[
\det\left[ \nabla^{2}u(x)\right]= \dfrac{f(x)}{\det \left[\nabla^{2}_{\xi} F(\nabla u(x))\right]}.
\]
Hence, by the well-known regularity results for fully nonlinear elliptic equations (see \cite{caca}), being $f\in C^{0,\alpha}(\Omega)$, and $f>0$, then $u\in C^{2,\alpha}(E_{u})$.
\end{rem}

The following comparison result holds.
\begin{theo}
Let $\Omega$ be a convex, bounded, open set in $\R^{2}$ with $C^{2}$ boundary, and let $u\in \Phi_0(\Omega)$ be a strong solution of problem \eqref{pb1det}. Consider the unique anisotropic radially symmetric strong solution $v$ of the symmetrized problem
\begin{equation}
\label{pb2det}
\left\{
\begin{array}{ll}
\det_H[v]=f^{\conv}(x) &\text{in } \Omega^{\stella} \\
v=0 &\text{on } \de\Omega^{\stella}.
\end{array}
\right.
\end{equation}
Then
\[
u^{\stella} \le v \text{ in }\Omega^{\stella}.
\]
\end{theo}
\begin{proof}
Let $u$ be a strong solution of problem \eqref{pb1det}. As observed in Remark \ref{reg}, $u\in C^{2}(E_{u})$. Hence, its level sets $\Omega_{t}$, $0\le t<\max_{\bar\Omega} u$ are $C^{2}$ and convex. Integrating both sides of the equation in \eqref{pb1det}, using \eqref{reilly}, the H\"older inequality, \eqref{gaussbonnet} and \eqref{proplambdaeq} we get
\begin{multline}
\int_{u>t}f(x)\,dx=\int_{u>t}{\det}_{H}[u]\, dx= \frac 1 2 \int_{u=t} k_{H}(x) \frac{H(\nabla u)^{3}}{|\nabla u|} d\mathcal H^{1}\ge \\ \ge\frac 1 2\dfrac{\ds\left(\int_{u=t} k_{H}(x)H(\nu)d\mathcal H^{1}\right)^{3} }
	{\ds\left(\int_{u=t}  \frac{k_{H}(x)}{|\nabla u|}d\mathcal H^{1}\right)^{2}} = \frac{4\kappa^{3}}{(-\lambda'_{H}(t))^{2}}  =4\kappa^{3} [(-u^{\aste})'(\lambda_{H}(t))]^{2},
\end{multline}
where last equality follows by the Definition \ref{asim} of symmetrization with respect to the anisotropic perimeter. By the well-known Hardy-Littlewood inequality and the anisotropic isoperimetric inequality \eqref{isop} we obtain
\[
[(-u^{\aste})'(\lambda_{H}(t))]^{2} \le \frac{1}{4\kappa^{3}} \int_{0}^{\mu(t)}f^{*}(r)\,dr  \le \frac{1}{4\kappa^{3}} \int_{0}^{\frac{ \lambda^{2}_{H}(t)}{4 \kappa}}f^{*}(r)\,dr.
\]
Here we mean $f^{*}(s)=0$ if $s\ge |\Omega|$. Performing the change of variable $s=\lambda_H(t)$ we get
\begin{equation}
\label{app}
[(-u^{\aste})'(s)]^{2}  \le \frac{1}{4\kappa^{3}} \int_{0}^{\frac{s^{2}}{4 \kappa}}f^{*}(r)\,dr, \quad s \in ]0,P_{H}(\Omega)].
\end{equation}
By \eqref{rad} the unique anisotropic radially symmetric strong solution to \eqref{pb2det}, $v(x)=w(r)$, with $r=H^{o}(x)$ is 
\[
w(r)=\frac{1}{\sqrt{\kappa}}\int_r^{R} \left(\int_{0}^{\kappa r^{2}} f^{*}(t) \,dt\right)^{\frac 1 2}\,dr,
\]
and then
\begin{equation}
\label{v}
v^{{\aste}}(s)=\dfrac{1}{2 \kappa^{{\frac 3 2}}}\int_{s}^{P_{H}(\Omega)}\bigg(\int_{0}^{\frac{\sigma^{2}}{4 \kappa}} f^{*}(t) \,dt\bigg)^{\frac 1 2}d\sigma \quad s \in [0,P_{H}(\Omega)].
\end{equation}
By \eqref{app} and \eqref{v} we get
\[
u^{{\aste}}(s) \le v^{\aste}(s), \quad s \in [0,P_{H}(\Omega)].
\]

 \end{proof}

\thanks{{\bf Acknowledgement.} This work has been partially supported by the FIRB 2013 project ``Geometrical and qualitative aspects of PDE's''}

\end{document}